\documentclass[12pt,electronic]{amsart}                                                 
\usepackage[T1]{fontenc}
\usepackage[applemac]{inputenc}
\usepackage[english]{babel}
\usepackage[small,nohug]{diagrams}
\usepackage{graphicx}
\usepackage{url}
\usepackage[hmargin=3.7cm,vmargin=3.5cm]{geometry} 
\usepackage{amsmath}
\usepackage{amsthm}                      
\usepackage{amssymb}
\usepackage{mathrsfs}
\usepackage{enumerate}
\theoremstyle{plain}                                                           
\newtheorem*{thm}{Theorem}
\newtheorem*{prop}{Proposition}
\theoremstyle{definition}
\newtheorem*{rem}{Remark}
\DeclareMathOperator{\Hom}{Hom}

\DeclareMathOperator{\Gal}{Gal}

\newcommand{\field}[1]{\ensuremath{\mathbf{#1}}}
\newcommand{\Q}{\ensuremath{\field{Q}}}      
\newcommand{\C}{\field C}  
\newcommand{\Z}{\ensuremath{\field{Z}}} 

\newcommand{\Ch}{\ensuremath{\mathrm{Chains}}}

\newcommand{\D}{{D}}
\newcommand{\PaB}{\mathbf{PaB}}
\newcommand{\CPaB}{\widehat\PaB}
\newcommand{\s}{\widetilde\sigma}
\newcommand{\Gm}{\mathbb{G}_m}
\renewcommand{\k}{\mathbf k}
\newcommand{\I}{\mathfrak I}

\title{Minimal models, GT-action and formality of the little disk operad}
\author{Dan Petersen}
\thanks{Supported by the G\"oran Gustafsson foundation for scientific and medical research}
\email{danpete@math.kth.se}
\address{Department of Mathematics \\ KTH Royal Institute of Technology \\ 100 44 Stockholm \\ Sweden}

\begin{document} 
  
 \maketitle

 \begin{abstract} We give a new proof of formality of the operad of little disks. The proof makes use of an operadic version of a simple formality criterion for commutative differential graded algebras due to Sullivan. We see that formality is a direct consequence of the fact that the Grothendieck--Teichm\"uller group operates on the chain operad of little disks.\end{abstract}

\section*{Introduction}

Let $D_2$ be the topological operad of little disks. It was proven in \cite{tamarkin} that this operad is \emph{formal}: there is a chain of quasi-isomorphisms of dg operads connecting the operad $\Ch(D_2)$ and its homology $H(D_2)$. A different proof, which works for little disks of any dimension,  was given in \cite{operadsmotives}, see also the improvements in \cite{lambrechtsvolic}. 

%In this note we give a new proof of formality of the operad $D_2$. This proof makes use of an operadic version of a simple formality criterion for commutative differential graded algebras due to Sullivan. We see that formality follows very simply from the fact that the Grothendieck--Teichm\"uller group acts on $\Ch(D_2)$.

%Since this note is quite short we do not give a detailed summary here. First we explain the aforementioned formality criterion, and explain how it applies to operads. We then use this criterion to prove formality of $D_2$. Finally we explain how this proof is related to Deligne's yoga of weights in cohomology.

In this note we give a short proof of formality of $D_2$. We begin by recalling from Sullivan a simple characterization of when a cdga is formal, and explain why this characterization carries over without changes to dg operads. The crucial tool is the notion of a minimal model of a cdga or a dg operad, respectively. Using this one can immediately deduce from the action of $GT$ on $\Ch(D_2)$ and the surjectivity of $GT(\Q) \to \Q^\times$, proven by Drinfel'd, that $D_2$ is a formal operad. Finally we give some motivation for the proof coming from the theory of weights in the cohomology of algebraic varieties.

I am grateful to Johan Alm for patient explanations, keen interest and stimulating conversations.
\section*{Formality of the little disk operad}

Fix a base field $\k$ of characteristic zero. If $V$ is a graded vector space, then we denote by $V^i$ its degree $i$ summand. We call $\phi_q \in \mathrm{GL}(V)$  a \emph{grading automorphism} if it has the form $\phi_q(v) = q^i v$ when $v \in V^i$, where $q \in \k^\times$ is a fixed non-root of unity. In the same way there are grading automorphisms of any graded algebra or any operad in graded vector spaces. % and assume that everything differential graded is nonnegatively graded with finite Betti numbers. 
The following proposition is proven in \cite[Theorem 12.7]{sullivaninfinitesimal}. We recall Sullivan's proof. 

\begin{prop}Let $A$ be a nilpotent commutative differential graded algebra. If a grading automorphism of $H(A)$ lifts to an automorphism of $A$, then $A$ is formal. \end{prop}
\begin{proof}Denote by $\sigma$ a lift to $A$ of the grading automorphism $\phi_q$ of $H(A)$. Let $p \colon M \stackrel \sim \to A$ be a minimal model. By comparing $p$ and $\sigma \circ p$, the uniqueness of the minimal model implies that $\sigma$ induces an automorphism $\s \colon M \to M$, well defined up to homotopy. 

From the explicit inductive construction of the minimal model one can see that the eigenvalues of $\s$ on $M^i$ are products of eigenvalues on $H^{i_n}(A)$, with $\sum_n i_n \geq i$. Thus all eigenvalues of $\s$ on $M^i$ have the form $q^j$, where $j \geq i$. Define $M_j$ as the subspace of $M$ where $\s$ acts as multiplication by $q^j$. Define 
\[ {\I} = \bigoplus_{j>i} M^i_j \qquad \text{and}\qquad S = \bigoplus_i M_i^i. \]
By the preceding paragraph we see that $M = \I \oplus S$, $dS = 0$, and that $\I$ is an ideal. Hence 
\[ M \to M/ (\I,d \I) = S/(S \cap d \I) = H(M)\]
makes sense and is easily seen to be a quasi-isomorphism. \end{proof}

We now assume that $P$ is a dg operad with $H(P)(0)=0$ and $H(P)(1) \cong \k$. This implies that $P$ has a minimal model, well defined up to homotopy, which may be constructed via an explicit inductive construction, see \cite{marklmodels}. In the next proposition we assume that $P$ is cohomologically graded, but the result is of course valid also in the homological case. That Sullivan's result is true for operads is also proven in \cite[Corollary 5.2.2]{formaloperads}. They, like Sullivan, use this result for proving that formality descends to a smaller ground field.

\begin{prop}If a grading automorphism of $H(P)$ lifts to $P$, then $P$ is formal. \end{prop}

\begin{proof}Repeat word for word the preceding proof, with the substitution $A \leadsto P$ and the tacit understanding that `minimal model' now refers to the operadic minimal model, and `ideal' refers to operadic ideal. \end{proof}

%We can now prove formality of the little disk operad $D_2$. We first recall very briefly the Grothendieck--Teichm\"uller group $GT$ and its action on $\Ch(D_2)$. See \cite{barnatan,tamarkin} or the expositions in \cite{merkulovgt,fressebook} for more details. One can identify $\Ch(D_2)$ with $\Ch (\mathrm{Nerve}(\widehat{\mathbf{PaB}}))$, where ${\mathbf{PaB}}$ is a combinatorially defined operad which is weakly equivalent to the operad of fundamental groupoids of $D_2$, and $\widehat{\mathbf{PaB}}$ denotes its pro-unipotent completion. An automorphism of $\CPaB$ is determined by a map $\PaB \to \CPaB$. The operad of groupoids $\PaB$ is generated by two morphisms $\tau$ in $\PaB(2)$ (the \emph{braiding}) and $\phi$ in  $\PaB(3)$ (the \emph{associator}), so we can describe automorphisms of $\CPaB$ by giving the images of $\tau$ and $\phi$. It turns out that one must map $\tau$ to a power $\tau^\lambda$, with $\lambda \in \k^\times$ -- this exponent makes sense, since we formed the pro-unipotent completion -- and that the image of $\phi$ can be described by an element $f$ in the pro-unipotent completion of the free group $F_2$, satisfying a certain list of equations which we do not write down. One can then define an algebraic group $GT$ consisting of all such pairs $(\lambda, f)$, with group operation corresponding to compositions of automorphisms. This is the Grothendieck--Teichm\"uller group. By construction it acts on $\CPaB$ and hence on $\Ch(D_2)$. 

We can now prove formality of the little disk operad $D_2$. We first recall very briefly the Grothendieck--Teichm\"uller group $GT$ and its action on $\Ch(D_2)$. See \cite{barnatan,tamarkin} or the expositions in \cite{merkulovgt,fressebook} for more details. 

There is an operad in groupoids $\PaB$, such that the objects of $\PaB(n)$ are parenthesized permutations of $\{1,\ldots,n\}$, and morphisms are braids on $n$ strands whose start and end must have the same label. There is a weak equivalence between $\PaB$ and the operad of fundamental groupoids of $\D_2$. Since moreover $D_2(n)$ is a $K(\pi,1)$ space for all $n$, we have an isomorphism $\Ch(D_2) \cong \Ch(\mathrm{Nerve}(\PaB)).$ If we take chains with $\k$-coefficients, then we may as well replace $\PaB$ with its $\k$-pro-unipotent completion $\CPaB$, as in rational homotopy theory. 
The completion is useful because whereas $\PaB$ itself does not have many automorphisms, it turns out that $\CPaB$ has a quite large automorphism group. 

\begin{figure}[h]
 \includegraphics[width=0.14\textwidth]{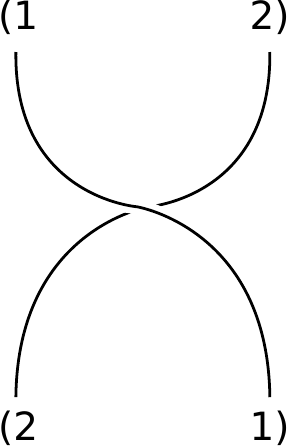} \hspace{0.15\textwidth}
 \includegraphics[width=0.25\textwidth]{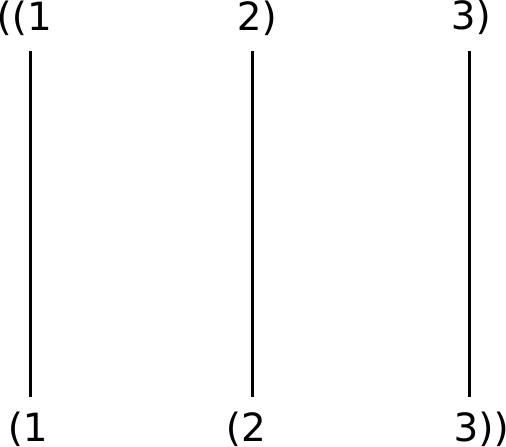}\hspace{0.15\textwidth}
 \includegraphics[width=0.14\textwidth]{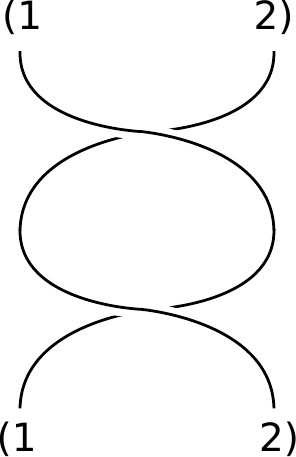}
 
 \caption{The braiding $\tau$, the associator $\phi$, and the twist $\tau^2$.}
 \label{tau}
 \end{figure}
The operad $\PaB$ is generated by a morphism $\tau$ in $\PaB(2)$ (the \emph{braiding}) and $\phi$ in  $\PaB(3)$ (the \emph{associator}), see Figure \ref{tau}, and an automorphism of $\CPaB$ is determined by the images of $\tau$ and $\phi$. The image of $\tau$ can be described by a scalar $\lambda \in \k^\times$: if we abusively denote by $\tau^2$ the `twist' in Figure \ref{tau}, then we must have $\tau^2 \mapsto (\tau^2)^\lambda$ for some such parameter, and $\lambda$ determines the image of $\tau$. The exponentiation makes sense because $\Hom_{\CPaB(2)}((12),(12))$ is a pro-unipotent group. Describing the image of $\phi$ is more complicated, since we need to describe an element of a completion of a \emph{three}-strand braid group. One finds that the image of $\phi$ can be described by an element $f$ in the pro-unipotent completion of the free group $F_2$, and that $f$ must satisfy a certain list of equations which we do not write down. One can then define an algebraic group $GT$ consisting of all such pairs $(\lambda, f)$, with group operation corresponding to compositions of automorphisms. This is the Grothendieck--Teichm\"uller group. By construction it acts on $\CPaB$ and hence on $\Ch(D_2)$.

\begin{thm}The operad $D_2$ of little disks is formal over $\Q$. \end{thm}

\begin{proof}Consider the map 
$ GT \to \Gm$ which maps a pair $(\lambda, f)$ to $\lambda$. We claim that this sends an automorphism of $\Ch(D_2)$ to the induced automorphism on homology, where $\Gm$ acts on homology via the grading action. The easiest way to see this is to use that the homology operad $H(D_2)$ (which is the operad of Gerstenhaber algebras) is generated in arity $2$. In particular the automorphism induced on homology by $(\lambda,f)$ can not depend on $f$, since $f$ only affects $\PaB(n)$ for $n \geq 3$. The space $D_2(2)$ is homotopic to a circle and its fundamental group is generated by the twist $\tau^2$. The map $\tau^2 \mapsto (\tau^2)^\lambda$ induces the identity on on $H_0(D_2(2))$ and multiplication by $\lambda$ on $H_1(D_2(2))$, which proves the claim. Finally,  $GT(\Q) \to \Q^\times$ is surjective (in fact even split), as proven in \cite[Section 5]{drinfeldGT}.
%In his original paper introducing the Grothendieck--Teichm\"uller group \cite{drinfeldGT}, Drinfel'd proved that $GT(\Q) \to \Q^\times$ is a split surjection. 
By the formality criterion established earlier, this shows that $D_2$ is formal. \end{proof}

\begin{rem}It is a well established principle that a formality isomorphism for the little disks must in one way or another involve the choice of an associator, see \cite{operadsmotives}. This principle holds true also for our proof: Drinfel'd deduces the surjectivity of $GT(\Q) \to \Q^\times$ from the existence of a rational associator. \end{rem}

\section*{Remarks on weights} Deligne, Griffiths, Morgan and Sullivan \cite{dgms} pro\-ved that compact K\"ahler manifolds are formal. Their proof uses classical Hodge theory and the $dd^c$-lemma. However, in the introduction they explain that they originally conjectured the result for smooth projective varieties by thinking about (at the time conjectural) properties of \'etale cohomology and positive characteristic algebraic geometry. Namely, one expected to be able to give purely algebraic constructions of Massey products in the \'etale cohomology, which should in particular be equivariant with respect to the Frobenius map. But the $n$th Massey product $\mu_n$ decreases cohomological degree by $n-2$, and by the Weil conjectures all eigenvalues of Frobenius on $H^i$ should have absolute value $q^{i/2}$. Thus Frobenius equivariance should force a `uniform' vanishing of $\mu_n$ for all $n > 2$, and we expect the variety to be formal. This is an instance of the philosophy of `weights' in cohomology, see e.g.\ Deligne's 1974 ICM address \cite{deligneicm}. 

A proof of formality along these lines was later obtained by Deligne via the proof of the Weil conjectures \cite[(5.3)]{weil2}: for $X$ a smooth complex projective variety, one may choose a countable subfield $\k$ over which $X$ is defined and use \'etale cohomology to obtain a dg algebra with an action of $\Gal(\overline \k / \k)$ computing $H(X)$, and the Galois action can be used to define a `weight filtration' which implies formality. 

The topological space $D_2(n)$ is homotopy equivalent to the configuration space of $n$ points in the complex plane. This, in turn, is the complex points of the algebraic variety 
\[ F_n = \mathbb A^n \setminus (\text{big diagonal})\]which is defined over $\Z$. This fact, as well as the actions of $\Gal(\overline \Q/\Q)$ on $\Ch(D_2) \otimes \Q_\ell$ for any prime $\ell$ (via the embedding $\Gal(\overline \Q/\Q) \hookrightarrow \widehat{GT}$ constructed in \cite{drinfeldGT} and  \cite{iharaembedding}, where $\widehat{GT}$ denotes the profinite version of the Grothendieck--Teichm\"uller group), can lead one to speculate that the operad $D_2$ is actually (up to homotopy) the base change to $\C$ of some algebro-geometrically defined operad defined over $\Q$ (or perhaps even $\Z$). This was proposed in \cite{morava}. Note though that the spaces $F_n$ do not themselves form an operad in any natural sense. The $\ell$-adic Galois representation on the \'etale cohomology group $H^i_{\text{\'et}}(F_n,\Q_\ell)$ is known: it is a sum of copies of the Tate object $\Q_\ell(-i)$ of weight $2i$, see \cite{kimhyperplane}. This coincides with the Galois action on $H^i(D_2(n)) \otimes \Q_\ell$ defined via $\widehat{GT}$, as one sees from the commutative diagram
%\[ \Gal(\overline \Q/\Q) \hookrightarrow \widehat{GT} \twoheadrightarrow \widehat{\Z}^\times \]
\begin{diagram}
\Gal(\overline \Q/\Q) &\rTo & \widehat{GT} & \rTo & \widehat{\Z}^\times \\
 & & \dTo & & \dTo \\
 & & GT(\Q_\ell) & \rTo & \Q_\ell^\times 
\end{diagram}
where the composition in the top row is the cyclotomic character.

We have explained that for smooth projective varieties the yoga of weights predicted vanishing of all Massey products. Something similar happens here. Suppose we did not know that $D_2$ is formal. By a Homotopy Transfer Theorem there is a structure of strong homotopy operad on $H(D_2)$ making it quasi-isomorphic to $\Ch(D_2)$ \cite{granaker}. Just as for $A_\infty$-algebras this structure is encoded by an infinite sequence of higher order multilinear operations $\mu_n$ which in this case raise homological degree by $n-2$. If these operations were compatible with the weights in cohomology, they would all need to vanish for $n >2$ and $D_2$ would be formal. 

In Deligne's formality proof we needed a Galois action to define the weight filtration, and the Galois action was obtained from \'etale cohomology. But here we do not need any algebraic geometry or a realization of Morava's proposal to get a Galois action on $\Ch(D_2)$, since we already know that $GT$ acts on this chain operad. All in all, this suggests strongly that there should exist a proof of formality of $D_2$ using only the fact that $GT$ acts on its operad of chains. The present note is the result of this line of thinking. 

\begin{rem}By reasoning with weights exactly as above, one is led to conjecture that operads of smooth projective varieties are always formal. In fact the main theorem of \cite{formaloperads} is that operads of compact K\"ahler manifolds are formal. Just as in \cite{dgms} their proof uses classical Hodge theory and does not directly involve the theory of weights. \end{rem}

\bibliographystyle{alpha}
\bibliography{../database}

\end{document}